\documentclass[11pt]{article}

\usepackage[margin=1.05in]{geometry}
\usepackage{amsmath,amssymb,amsthm,mathtools}
\usepackage{microtype}
\usepackage{enumitem}
\usepackage{booktabs}
\usepackage{hyperref}
\usepackage[nameinlink,capitalize,noabbrev]{cleveref}

\hypersetup{colorlinks=true,linkcolor=blue,citecolor=blue,urlcolor=blue}
\setlist[itemize]{leftmargin=2em,itemsep=0.25em,topsep=0.35em}
\setlist[enumerate]{leftmargin=2em,itemsep=0.25em,topsep=0.35em}

\newtheorem{theorem}{Theorem}[section]
\newtheorem{proposition}[theorem]{Proposition}
\newtheorem{corollary}[theorem]{Corollary}

\theoremstyle{definition}
\newtheorem{definition}[theorem]{Definition}
\newtheorem{conjecture}[theorem]{Conjecture}
\theoremstyle{remark}
\newtheorem{remark}[theorem]{Remark}

\crefname{theorem}{Theorem}{Theorems}
\Crefname{theorem}{Theorem}{Theorems}
\crefname{proposition}{Proposition}{Propositions}
\Crefname{proposition}{Proposition}{Propositions}
\crefname{corollary}{Corollary}{Corollaries}
\Crefname{corollary}{Corollary}{Corollaries}
\crefname{definition}{Definition}{Definitions}
\Crefname{definition}{Definition}{Definitions}
\crefname{conjecture}{Conjecture}{Conjectures}
\Crefname{conjecture}{Conjecture}{Conjectures}
\crefname{remark}{Remark}{Remarks}
\Crefname{remark}{Remark}{Remarks}

\newcommand{\Sph}{\mathbb S}
\newcommand{\R}{\mathbb R}
\newcommand{\K}{\mathcal K}
\newcommand{\Harm}{\mathcal H}

\newcommand{\MA}{\mathrm{MA}}
\newcommand{\Area}{\operatorname{Area}}
\newcommand{\Vol}{\operatorname{Vol}}
\newcommand{\Lip}{\operatorname{Lip}}
\newcommand{\tr}{\operatorname{tr}}

\newcommand{\eps}{\varepsilon}

\title{\textbf{Minkowski Polytopes of Spherical Designs:}\\[3pt]
\textbf{High-Order Isotropy and Quantitative Sphericity}}
\author{Congpei An\\[4pt]
\small 
\small \\[2pt]
\small  \href{mailto:andbachcp@gmail.com}{andbachcp@gmail.com}\\
\small }
\date{}

\begin{document}
\maketitle

\begin{abstract}
Let $X_N=\{x_1,\dots,x_N\}\subset \Sph^2$ be a spherical $t$-design of strength $t\ge2$, and let $\mu_X$ be its empirical measure. Minkowski's theorem associates with $X_N$ a convex polytope $P_X\subset\R^3$, unique up to translation, whose facet normals are the design nodes and whose facets all have area $4\pi/N$; equivalently, $S_{P_X}=4\pi\mu_X$. We show that this realization transfers polynomial exactness into exact convex geometry: the normalized surface tensors of $P_X$ agree with those of the unit ball through order $t$, and mixed volumes are exact against convex bodies whose support functions are spherical polynomials of degree at most $t$.

We next derive quantitative shape information. Spherical Jackson approximation yields a $1$-Wasserstein discrepancy $W_1(\mu_X,\sigma)=O(t^{-1})$, while degree-two exactness gives a uniform nondegeneracy condition. Combined with quantitative inverse stability for Minkowski's problem, this implies, after Steiner normalization,
\[
d_H(P_X,B)=O(t^{-1/2}),\qquad \alpha(P_X,B)=O(t^{-3/4}),
\]
for every spherical $t$-design, without assumptions on cardinality, separation, covering radius, or spectral conditioning. Projection bodies retain the full $O(t^{-1})$ scale, separating linear surface-area observables from nonlinear reconstruction of the body. In the critical regime $N=(t+1)^2$, a uniform spectral lower bound for the sampling Gram matrix further forces the facet normals to be separated at the wavelength scale $t^{-1}$. The construction extends to $\Sph^d$, with the universal Hausdorff rate $O(t^{-1/d})$.
\end{abstract}



\section{Introduction}

A spherical $t$-design is a finite set of points on a sphere whose equal-weight average reproduces spherical integration for all polynomials of degree at most $t$. Since their introduction by Delsarte, Goethals, and Seidel~\cite{DGS1977}, spherical designs have been studied in algebraic combinatorics, approximation and cubature, energy and covering problems, and numerical construction; see, for example, Bannai--Bannai~\cite{BannaiBannai2009}, Sloan--Womersley~\cite{SloanWomersley2009}, and An--Chen--Sloan--Womersley~\cite{An2010}. Existence for arbitrary strength goes back to Seymour and Zaslavsky~\cite{SeymourZaslavsky1984}, while Bondarenko, Radchenko, and Viazovska proved the optimal-order existence result $N=O(t^d)$ on $\Sph^d$~\cite{BRV2013}.

The purpose of this paper is to extract a different consequence of the design equations: their convex geometry. Let
\[
\mu_X=\frac1N\sum_{j=1}^N\delta_{x_j},
\qquad
\sigma=\frac{\omega}{4\pi},
\]
where $\omega$ denotes surface area measure on $\Sph^2$. For a spherical $t$-design with $t\ge2$, exactness in degrees one and two gives
\begin{equation}\label{eq:intro-moments}
\int_{\Sph^2}u\,d\mu_X(u)=0,
\qquad
\int_{\Sph^2}uu^T\,d\mu_X(u)=\frac13 I_3.
\end{equation}
The first identity is the balance condition in Minkowski's existence theorem, and the second prevents concentration on a great circle. Consequently there is a convex polytope $P_X\subset\R^3$, unique up to translation, such that
\begin{equation}\label{eq:intro-surface-measure}
S_{P_X}=\frac{4\pi}{N}\sum_{j=1}^N\delta_{x_j}=4\pi\mu_X.
\end{equation}
Thus the nodes of the design are precisely the outer facet normals of an equal-facet-area polytope.

The existence statement in \eqref{eq:intro-surface-measure} is an immediate application of the classical Minkowski theorem. The main point of the paper is what the remaining design equations force after this realization. The normalized surface-area measure of $P_X$ agrees with that of the ball on every polynomial of degree at most $t$. Hence the algebraic exactness of a design becomes a hierarchy of exact geometric identities for the associated convex body.

Our first main result makes this principle explicit. If $M_k(K)$ denotes the normalized $k$th surface tensor of a convex body, then
\begin{equation}\label{eq:intro-isotropy}
M_k(P_X)=M_k(B),\qquad 0\le k\le t.
\end{equation}
Thus the design strength measures the order to which the normal geometry of $P_X$ is isotropic. The same mechanism applies to mixed volumes: whenever the support function of a convex body $L$ belongs to $\Pi_t(\Sph^2)$,
\[
V(P_X,P_X,L)=V(B,B,L).
\]
This viewpoint is related to the use of moments and surface tensors to encode convex bodies; see Kousholt--Kiderlen~\cite{KousholtKiderlen2016}, Kousholt~\cite{Kousholt2017}, and Kousholt--Schulte~\cite{KousholtSchulte2021}. Here, however, the moment data arise simultaneously and exactly from the design equations and determine a distinguished equal-area polytope.

The second part of the paper is quantitative. By combining polynomial exactness with spherical Jackson approximation and Kantorovich--Rubinstein duality, we prove
\begin{equation}\label{eq:intro-w1}
W_1(\mu_X,\sigma)\le \frac{C}{t+1}.
\end{equation}
At the same time, \eqref{eq:intro-moments} implies the uniform dispersion estimate
\[
\Theta(\mu_X):=\inf_{\theta\in\Sph^2}\int_{\Sph^2}|\theta\cdot u|\,d\mu_X(u)\ge\frac13.
\]
Thus every spherical design lies in a uniformly nondegenerate class of surface-area measures. We then invoke the recent quantitative inverse-stability theorem of K.~J. B\"or\"oczky, J.~M. Machado, and J.~P.~G. Ramos, \emph{Quantitative stability for Minkowski's problem}, arXiv:2603.17726v3 (2026), to obtain
\begin{equation}\label{eq:intro-hausdorff}
\inf_{a\in\R^3}d_H(P_X,a+B)\le C(t+1)^{-1/2},
\end{equation}
and hence, after fixing translations by the Steiner point,
\[
d_H(P_X,B)\le C(t+1)^{-1/2}.
\]
The estimate is universal: it requires no assumption on cardinality, separation, covering radius, mesh ratio, or spectral conditioning. The external inverse-stability theorem is an input; the contribution here is to show that the two hypotheses relevant for its application---an $O(t^{-1})$ discrepancy and a uniform nondegeneracy bound---are consequences of design exactness itself.

The passage from \eqref{eq:intro-w1} to \eqref{eq:intro-hausdorff} also reveals a useful distinction. The surface-area data are already approximated at the scale $t^{-1}$, whereas reconstruction of the body through the inverse Minkowski map yields the universal scale $t^{-1/2}$ in dimension three. For projection bodies the dependence on the surface-area measure is linear, and the full rate is retained:
\begin{equation}\label{eq:intro-projection}
d_H(\Pi P_X,\pi B)\le C(t+1)^{-1}.
\end{equation}
Thus the two rates separate approximation of linear surface-area observables from nonlinear shape reconstruction.

A third theme concerns local geometry. Polynomial exactness alone does not control the minimum separation of the nodes. In the critical interpolation regime $N=(t+1)^2$, however, a uniform lower spectral bound for the normalized sampling Gram matrix implies
\[
\min_{i\ne j}d_{\Sph^2}(x_i,x_j)\gtrsim t^{-1}.
\]
For the Minkowski polytope this says that the facet normals are separated at the natural wavelength scale. Algebraic exactness and spectral conditioning therefore play different roles: the former controls universal global sphericity, while the latter supplies finer local regularity. This leads naturally to the question of whether well-conditioned critical designs enjoy a shape rate strictly better than the universal exponent $1/2$.

The surface-area measure also admits an Aleksandrov Minkowski--Monge--Amp\`ere interpretation. We use it only as a geometric language for harmonic shape modes. Linearization at the ball gives the operator $\Delta_{\Sph^2}+2$, whose kernel consists of translations and whose degree-two modes describe infinitesimal ellipsoidal anisotropy. The global results, however, are obtained from harmonic approximation and convex-geometric stability rather than from local PDE analysis.

In summary, the paper develops the chain
\[
\text{spherical harmonic exactness}
\longrightarrow
\text{exact surface-area geometry}
\longrightarrow
\text{quantitative sphericity},
\]
with spectral conditioning providing a separate local refinement. The same mechanism extends to $\Sph^d\subset\R^{d+1}$ and yields the universal rate $d_H(P_X,B^{d+1})=O(t^{-1/d})$ after normalization.

The paper is organized as follows. \Cref{sec:prelim} fixes notation. \Cref{sec:realization} establishes the equal-area Minkowski realization. \Cref{sec:isotropy} develops the exact surface-tensor and mixed-volume identities. \Cref{sec:quantitative} proves the Wasserstein and universal shape estimates, and \Cref{sec:projection} treats projection bodies. \Cref{sec:modes} records the complementary harmonic shape-mode interpretation. \Cref{sec:affine} gives the affine covariance statement. \Cref{sec:spectral} studies spectral conditioning and wavelength-scale separation. \Cref{sec:higher} contains the higher-dimensional extension and open problems.

\section{Preliminaries}\label{sec:prelim}

\subsection{Spherical designs}
Let $\omega$ denote the standard surface area measure on $\Sph^2$ and let
\[
\sigma:=\frac{\omega}{4\pi}
\]
be normalized area measure. Let $\Pi_t(\Sph^2)$ denote the restrictions to $\Sph^2$ of polynomials in three variables of total degree at most $t$.

\begin{definition}[Spherical design]\label{def:design}
A finite set $X_N=\{x_1,\dots,x_N\}\subset\Sph^2$ of distinct points is a spherical $t$-design if
\begin{equation}\label{eq:design}
\frac1N\sum_{j=1}^N p(x_j)=\int_{\Sph^2}p(u)\,d\sigma(u)
\qquad\text{for every }p\in\Pi_t(\Sph^2).
\end{equation}
Its empirical measure is
\[
\mu_X:=\frac1N\sum_{j=1}^N\delta_{x_j}.
\]
\end{definition}

Let $\Harm_\ell$ be the space of spherical harmonics of degree $\ell$. The classical harmonic characterization of a spherical design is
\begin{equation}\label{eq:harmonic-design}
\int_{\Sph^2}Y\,d\mu_X=0
\qquad\text{for every }Y\in\Harm_\ell,\quad 1\le\ell\le t.
\end{equation}
See Delsarte--Goethals--Seidel~\cite{DGS1977}.

Two moment identities will be used repeatedly. If $t\ge1$, then
\begin{equation}\label{eq:first-moment}
\int_{\Sph^2}u\,d\mu_X(u)=0.
\end{equation}
If $t\ge2$, then
\begin{equation}\label{eq:second-moment}
\int_{\Sph^2}uu^T\,d\mu_X(u)=\frac13 I_3.
\end{equation}
The second formula follows from rotation invariance and $\tr(uu^T)=1$.

\subsection{Convex bodies and surface area measures}
Let $\K^3$ be the family of nonempty compact convex subsets of $\R^3$, and let $\K^3_\circ$ denote the convex bodies with nonempty interior. The support function of $K\in\K^3$ is
\[
h_K(u):=\max_{x\in K}x\cdot u,\qquad u\in\Sph^2.
\]
Translations act by
\[
h_{K+a}(u)=h_K(u)+a\cdot u.
\]
The Hausdorff metric satisfies
\begin{equation}\label{eq:hausdorff-support}
d_H(K,L)=\|h_K-h_L\|_{L^\infty(\Sph^2)}.
\end{equation}

The surface area measure $S_K$ is the push-forward of surface area on $\partial K$ under the outer unit normal map. For a polytope $P$ with facets $F_1,\dots,F_M$ and corresponding outer unit normals $u_1,\dots,u_M$,
\begin{equation}\label{eq:polytope-SA}
S_P=\sum_{j=1}^M\Area(F_j)\delta_{u_j}.
\end{equation}
Its total mass is
\[
S_K(\Sph^2)=\Area(\partial K).
\]

We use the following classical form of Minkowski's existence and uniqueness theorem; see Schneider~\cite[Section~8.2]{Schneider2014}.

\begin{theorem}[Minkowski]\label{thm:minkowski}
Let $\nu$ be a finite nonzero Borel measure on $\Sph^2$. There exists a convex body $K\in\K^3_\circ$ satisfying $S_K=\nu$ if and only if
\begin{equation}\label{eq:minkowski-balance}
\int_{\Sph^2}u\,d\nu(u)=0
\end{equation}
and $\nu$ is not concentrated on any great circle. The body $K$ is unique up to translation. If $\nu$ has finite support, then $K$ is a polytope.
\end{theorem}

\subsection{Mixed volumes and projection functions}
For convex bodies $K,L\in\K^3$, the first variation of volume is
\[
\Vol(K+sL)=\Vol(K)+3sV(K,K,L)+O(s^2),\qquad s\downarrow0.
\]
The mixed volume has the integral representation
\begin{equation}\label{eq:mixed-volume}
V(K,K,L)=\frac13\int_{\Sph^2}h_L(u)\,dS_K(u).
\end{equation}
The brightness or projection function of $K$ is
\[
b_K(v):=\Area(K|v^\perp),\qquad v\in\Sph^2.
\]
Cauchy's projection formula is
\begin{equation}\label{eq:cauchy}
b_K(v)=\frac12\int_{\Sph^2}|u\cdot v|\,dS_K(u).
\end{equation}
The projection body $\Pi K$ is the origin-symmetric convex body with support function
\begin{equation}\label{eq:projection-body}
h_{\Pi K}(v)=b_K(v).
\end{equation}
See Gardner~\cite{Gardner2006} or Hug--Weil~\cite{HugWeil2020}.

\section{Equal-area Minkowski realizations}\label{sec:realization}

We first isolate the basic geometric correspondence. It is this realization, rather than the Monge--Amp\`ere formulation, that serves as the structural foundation of the paper.

\begin{theorem}[Minkowski realization of a spherical design]\label{thm:realization}
Let $X_N=\{x_1,\dots,x_N\}\subset\Sph^2$ be a spherical $t$-design with $t\ge2$. Then there exists a convex polytope $P_X\in\K^3_\circ$, unique up to translation, such that
\begin{equation}\label{eq:realization}
S_{P_X}=4\pi\mu_X=\frac{4\pi}{N}\sum_{j=1}^N\delta_{x_j}.
\end{equation}
Consequently, $P_X$ has exactly $N$ facets, the outer unit normal of the $j$th facet is $x_j$, every facet has area $4\pi/N$, and
\begin{equation}\label{eq:surface-total}
\Area(\partial P_X)=4\pi.
\end{equation}
\end{theorem}

\begin{proof}
Set
\[
\nu_X:=4\pi\mu_X=\frac{4\pi}{N}\sum_{j=1}^N\delta_{x_j}.
\]
By \eqref{eq:first-moment},
\[
\int_{\Sph^2}u\,d\nu_X(u)=\frac{4\pi}{N}\sum_{j=1}^Nx_j=0.
\]
It remains to show that $\nu_X$ is not concentrated on a great circle. Suppose that all $x_j$ lie in $a^\perp\cap\Sph^2$ for some $a\ne0$. Then \eqref{eq:second-moment} gives
\[
0=\frac1N\sum_{j=1}^N(a\cdot x_j)^2
=a^T\left(\frac13I_3\right)a
=\frac13|a|^2,
\]
a contradiction. Hence \Cref{thm:minkowski} gives a convex body $P_X$, unique up to translation, with $S_{P_X}=\nu_X$. Since $\nu_X$ has finite support, $P_X$ is a polytope. By \eqref{eq:polytope-SA}, each atom gives one facet with the stated normal and area. Summing the masses gives \eqref{eq:surface-total}.
\end{proof}

\begin{definition}[Minkowski polytope of a design]\label{def:minkowski-polytope}
The translation class determined by \Cref{thm:realization} is called the \emph{Minkowski polytope} of $X_N$. When a fixed representative is required, we impose the Steiner normalization
\begin{equation}\label{eq:steiner-normalization}
s(P_X)=0,
\end{equation}
where $s(K)$ denotes the Steiner point of $K$.
\end{definition}

\begin{remark}[The case $t=1$]
Degree-one exactness supplies the balance condition but does not prevent lower-dimensional support. For example, two antipodal points form a spherical $1$-design. The conclusion of \Cref{thm:realization} remains valid for a spherical $1$-design if one additionally assumes that the nodes are not contained in a great circle, equivalently that they span $\R^3$.
\end{remark}

The construction has a precise converse.

\begin{theorem}[Equal-area polytope characterization]\label{thm:converse}
Let $P\in\K^3_\circ$ be a convex polytope with $N$ facets of equal area, and let $x_1,\dots,x_N$ be their outer unit normals. Then
\begin{equation}\label{eq:normalized-SP}
\frac{S_P}{S_P(\Sph^2)}=\frac1N\sum_{j=1}^N\delta_{x_j}.
\end{equation}
Consequently, $X_N=\{x_1,\dots,x_N\}$ is a spherical $t$-design if and only if
\begin{equation}\label{eq:polytope-char}
\frac1{S_P(\Sph^2)}\int_{\Sph^2}p(u)\,dS_P(u)
=\int_{\Sph^2}p(u)\,d\sigma(u)
\end{equation}
for every $p\in\Pi_t(\Sph^2)$.
\end{theorem}

\begin{proof}
If the common facet area is $a>0$, then
\[
S_P=a\sum_{j=1}^N\delta_{x_j},\qquad S_P(\Sph^2)=Na,
\]
which proves \eqref{eq:normalized-SP}. Substitution into \eqref{eq:polytope-char} gives exactly \eqref{eq:design}.
\end{proof}

\begin{corollary}[Orthogonal symmetries lift to polytope symmetries]\label{cor:symmetry}
Let $X_N$ be as in \Cref{thm:realization}, and let $Q\in O(3)$ satisfy $QX_N=X_N$. For the Steiner-normalized representative of $P_X$,
\[
QP_X=P_X.
\]
\end{corollary}

\begin{proof}
The equality $QX_N=X_N$ implies $Q_\#S_{P_X}=S_{P_X}$, while $S_{QP_X}=Q_\#S_{P_X}$. Minkowski uniqueness gives $QP_X=P_X+a$ for some $a\in\R^3$. The Steiner point is translation covariant and orthogonally equivariant, so the normalization $s(P_X)=0$ gives $s(QP_X)=0=s(P_X+a)=a$. Hence $a=0$.
\end{proof}

\section{High-order isotropy and exact geometric identities}\label{sec:isotropy}

The Minkowski realization turns polynomial exactness into exact identities for the normal geometry of $P_X$. This is the first major consequence of the correspondence.

\subsection{Surface tensors}

\begin{definition}[Normalized surface tensor]\label{def:surface-tensor}
For $K\in\K^3_\circ$ and $k\ge0$, define
\begin{equation}\label{eq:surface-tensor}
M_k(K):=\frac1{S_K(\Sph^2)}\int_{\Sph^2}u^{\otimes k}\,dS_K(u).
\end{equation}
\end{definition}
These tensors are translation invariant and orthogonally covariant:
\[
M_k(QK)=Q^{\otimes k}M_k(K),\qquad Q\in O(3).
\]

\begin{theorem}[High-order isotropy]\label{thm:isotropy}
Let $X_N$ be a spherical $t$-design with $t\ge2$, and let $P_X$ be its Minkowski polytope. Then
\begin{equation}\label{eq:isotropy}
M_k(P_X)=M_k(B),\qquad 0\le k\le t.
\end{equation}
Equivalently, for every $a\in\R^3$ and every integer $m$ with $2m\le t$,
\begin{equation}\label{eq:even-moments}
\frac1{4\pi}\int_{\Sph^2}(a\cdot u)^{2m}\,dS_{P_X}(u)=\frac{|a|^{2m}}{2m+1},
\end{equation}
and all odd moments of order at most $t$ vanish.
\end{theorem}

\begin{proof}
By \eqref{eq:realization},
\[
M_k(P_X)=\int_{\Sph^2}u^{\otimes k}\,d\mu_X(u).
\]
Every component of $u^{\otimes k}$ is a polynomial of degree $k\le t$, so design exactness gives
\[
M_k(P_X)=\int_{\Sph^2}u^{\otimes k}\,d\sigma(u)=M_k(B).
\]
For fixed $a$, apply this identity to $(a\cdot u)^k$. After rotating $a/|a|$ to the north pole,
\[
\int_{\Sph^2}(a\cdot u)^{2m}\,d\sigma(u)
=|a|^{2m}\frac12\int_{-1}^1s^{2m}\,ds
=\frac{|a|^{2m}}{2m+1}.
\]
Odd moments vanish by antipodal symmetry of $\sigma$.
\end{proof}

Let $\mathcal P_2(2m)$ denote the set of pair partitions of $\{1,\dots,2m\}$.

\begin{corollary}[Explicit isotropic tensors]\label{cor:explicit-tensors}
If $2m\le t$, then
\begin{equation}\label{eq:explicit-tensor}
[M_{2m}(P_X)]_{i_1\cdots i_{2m}}
=\frac1{(2m+1)!!}
\sum_{\pi\in\mathcal P_2(2m)}\prod_{\{a,b\}\in\pi}\delta_{i_ai_b}.
\end{equation}
In particular,
\[
M_1(P_X)=0,\qquad M_2(P_X)=\frac13I_3,
\]
and
\[
[M_4(P_X)]_{ijkl}
=\frac1{15}(\delta_{ij}\delta_{kl}+\delta_{ik}\delta_{jl}+\delta_{il}\delta_{jk}).
\]
\end{corollary}

\begin{proof}
The uniform moment tensor is rotationally invariant and hence a scalar multiple of the complete symmetrization of tensor products of the Euclidean metric. Contracting all indices with a vector $a$ and using \eqref{eq:even-moments} determines the scalar as $1/(2m+1)!!$.
\end{proof}

\subsection{Mixed volumes}
The same exactness principle applies to geometric functionals that are linear in the surface-area measure.

\begin{theorem}[Exact mixed volumes for polynomial support functions]\label{thm:mixed-exact}
Let $X_N$ be a spherical $t$-design and let $P_X$ be its Minkowski polytope. If $L\in\K^3$ has support function $h_L\in\Pi_t(\Sph^2)$, then
\begin{equation}\label{eq:mixed-exact}
V(P_X,P_X,L)=V(B,B,L).
\end{equation}
More generally, for every $L\in\K^3$,
\begin{equation}\label{eq:mixed-error}
|V(P_X,P_X,L)-V(B,B,L)|\le\frac{8\pi}{3}E_t(h_L)_\infty,
\end{equation}
where
\[
E_t(h_L)_\infty:=\inf_{p\in\Pi_t(\Sph^2)}\|h_L-p\|_\infty.
\]
\end{theorem}

\begin{proof}
Using \eqref{eq:mixed-volume}, \eqref{eq:realization}, and $S_B=4\pi\sigma$,
\begin{equation}\label{eq:mixed-difference}
V(P_X,P_X,L)-V(B,B,L)
=\frac{4\pi}{3}\int_{\Sph^2}h_L\,d(\mu_X-\sigma).
\end{equation}
If $h_L\in\Pi_t$, the right-hand side vanishes. For arbitrary $p\in\Pi_t$, exactness gives
\[
\left|\int h_L\,d(\mu_X-\sigma)\right|
=\left|\int(h_L-p)\,d(\mu_X-\sigma)\right|
\le2\|h_L-p\|_\infty.
\]
Taking the infimum proves \eqref{eq:mixed-error}.
\end{proof}

\begin{corollary}[Lipschitz mixed-volume estimate]\label{cor:mixed-lip}
Let $C_J$ be a Jackson constant on $\Sph^2$: every Lipschitz function $f$ admits $p_t\in\Pi_t$ with
\begin{equation}\label{eq:jackson}
\|f-p_t\|_\infty\le\frac{C_J}{t+1}\Lip(f).
\end{equation}
If, after a translation, $L\subset RB$, then
\begin{equation}\label{eq:mixed-lip}
|V(P_X,P_X,L)-V(B,B,L)|\le\frac{8\pi C_JR}{3(t+1)}.
\end{equation}
\end{corollary}

\begin{proof}
For $u,v\in\Sph^2$,
\[
|h_L(u)-h_L(v)|\le R|u-v|\le R\,d_{\Sph^2}(u,v),
\]
so $\Lip(h_L)\le R$. Combine \eqref{eq:jackson} with \eqref{eq:mixed-error}. The spherical Jackson inequality can be found, for example, in Reimer~\cite[Corollary~6.37]{Reimer2003} or Dai--Xu~\cite{DaiXu2013}.
\end{proof}

\begin{remark}[Geometric meaning of design strength]\label{rem:geometric-strength}
\Cref{thm:isotropy,thm:mixed-exact} give a geometric interpretation of the parameter $t$. Increasing the design strength does not merely enlarge a polynomial test space: it forces the normalized surface-area measure of $P_X$ to reproduce progressively higher-order isotropic tensors and a progressively larger family of exact mixed-volume functionals. In this sense, spherical $t$-designs encode high-order normal isotropy of equal-area polytopes.
\end{remark}

\section{Wasserstein discrepancy and universal quantitative sphericity}\label{sec:quantitative}

We now pass from exact identities to quantitative shape control. The argument has three ingredients: polynomial exactness gives a $W_1$ estimate, degree-two exactness gives a uniform dispersion bound, and quantitative inverse stability for Minkowski's theorem converts these measure-theoretic statements into Hausdorff control.

\subsection{Wasserstein discrepancy from polynomial exactness}
We use the spherical Jackson estimate \eqref{eq:jackson}.

\begin{theorem}[Universal Wasserstein bound]\label{thm:w1}
Every spherical $t$-design satisfies
\begin{equation}\label{eq:w1}
W_1(\mu_X,\sigma)\le\frac{2C_J}{t+1}.
\end{equation}
Consequently,
\begin{equation}\label{eq:w1-surface}
W_1\!\left(\frac{S_{P_X}}{S_{P_X}(\Sph^2)},\frac{S_B}{S_B(\Sph^2)}\right)
\le\frac{2C_J}{t+1}.
\end{equation}
\end{theorem}

\begin{proof}
By Kantorovich--Rubinstein duality,
\[
W_1(\mu_X,\sigma)
=\sup_{\Lip(f)\le1}\left|\int_{\Sph^2}f\,d(\mu_X-\sigma)\right|.
\]
For such $f$, choose $p_t$ as in \eqref{eq:jackson}. Since the design integrates $p_t$ exactly,
\[
\left|\int f\,d(\mu_X-\sigma)\right|
=\left|\int(f-p_t)\,d(\mu_X-\sigma)\right|
\le2\|f-p_t\|_\infty
\le\frac{2C_J}{t+1}.
\]
Taking the supremum proves \eqref{eq:w1}. The surface-area form follows from $S_{P_X}/4\pi=\mu_X$ and $S_B/4\pi=\sigma$.
\end{proof}

\subsection{Uniform dispersion forced by degree-two exactness}
For a probability measure $\mu$ on $\Sph^2$, define
\begin{equation}\label{eq:theta}
\Theta(\mu):=\inf_{\theta\in\Sph^2}\int_{\Sph^2}|\theta\cdot u|\,d\mu(u).
\end{equation}
This functional is positive exactly when the measure is not supported on a great circle and is the quantitative nondegeneracy parameter used in quantitative stability theory for Minkowski's problem.

\begin{proposition}[Designs are uniformly dispersed]\label{prop:dispersion}
If $X_N$ is a spherical $t$-design with $t\ge2$, then
\begin{equation}\label{eq:dispersion}
\Theta(\mu_X)\ge\frac13.
\end{equation}
Moreover,
\[
\Theta(\sigma)=\frac12.
\]
\end{proposition}

\begin{proof}
Fix $\theta\in\Sph^2$. Since $|s|\ge s^2$ for $|s|\le1$,
\[
\int|\theta\cdot u|\,d\mu_X(u)
\ge\int(\theta\cdot u)^2\,d\mu_X(u)
=\theta^T\left(\frac13I_3\right)\theta
=\frac13,
\]
where the second-moment identity follows from degree-two exactness. Taking the infimum proves \eqref{eq:dispersion}. For $\sigma$, rotational invariance reduces the integral to $\frac12\int_{-1}^1|s|\,ds=1/2$.
\end{proof}

\begin{remark}\label{rem:dispersion-universal}
The estimate \eqref{eq:dispersion} is independent of $N$ and of any separation or mesh-ratio condition. It is a direct convex-geometric consequence of the isotropic second moment. Thus the nondegeneracy needed for quantitative inversion of the surface-area measure is already built into every spherical design of strength at least two.
\end{remark}

\subsection{Quantitative inverse stability}
For probability measures $\mu,\nu$ on $\Sph^2$, define the dual-convex distance
\begin{equation}\label{eq:dc}
d_C(\mu,\nu):=\sup_{K\subset B}\left|\int_{\Sph^2}h_K\,d(\mu-\nu)\right|,
\end{equation}
where the supremum runs over convex bodies contained in the unit ball. Since $h_K$ is $1$-Lipschitz when $K\subset B$,
\begin{equation}\label{eq:dc-w1}
d_C(\mu,\nu)\le W_1(\mu,\nu).
\end{equation}
For convex bodies $K,L\subset\R^n$, let
\begin{equation}\label{eq:fraenkel}
\alpha(K,L):=\inf_{\substack{a\in\R^n,\ r>0\\ r^n|L|=|K|}}
\frac{|K\triangle(a+rL)|}{|K|}.
\end{equation}

We record the form of the quantitative inverse Minkowski theorem used below. It is Theorem~1.1 of K.~J. B\"or\"oczky, J.~M. Machado, and J.~P.~G. Ramos, \emph{Quantitative stability for Minkowski's problem}, arXiv:2603.17726v3 (2026).

\begin{theorem}[Quantitative stability for Minkowski's problem]\label{thm:BMR}
Let $\mu$ and $\nu$ be centered probability measures on $\Sph^{n-1}$, and suppose $\Theta(\mu),\Theta(\nu)\ge\vartheta>0$. Let $E_\mu,E_\nu$ be convex bodies whose surface area measures are $\mu$ and $\nu$, respectively. Then
\begin{equation}\label{eq:BMR-H}
\inf_{a\in\R^n}d_H(E_\mu,a+E_\nu)
\le C_{n,\vartheta}d_C(\mu,\nu)^{1/(n-1)}.
\end{equation}
Furthermore,
\begin{equation}\label{eq:BMR-F}
\alpha(E_\mu,E_\nu)^2
\le C_{n,\vartheta}d_C(\mu,\nu)^{1+1/(n-1)}.
\end{equation}
The exponent $1/(n-1)$ in \eqref{eq:BMR-H} is sharp in the general class covered by the theorem.
\end{theorem}

\begin{remark}[Normalization]\label{rem:normalization}
The theorem is stated for probability surface area measures. In $\R^3$, surface area measures satisfy $S_{rK}=r^2S_K$. Hence if
\[
r_0=(4\pi)^{-1/2},
\]
then
\begin{equation}\label{eq:prob-normalization}
S_{r_0P_X}=\mu_X,\qquad S_{r_0B}=\sigma.
\end{equation}
This elementary rescaling is the only normalization needed below.
\end{remark}

\subsection{Main quantitative sphericity theorem}

\begin{theorem}[Universal quantitative sphericity]\label{thm:quantitative-sphericity}
There is an absolute constant $C>0$ such that every spherical $t$-design $X_N\subset\Sph^2$ with $t\ge2$ and Minkowski polytope $P_X$ satisfies
\begin{equation}\label{eq:main-H}
\inf_{a\in\R^3}d_H(P_X,a+B)\le\frac{C}{\sqrt{t+1}}.
\end{equation}
If $P_X$ is normalized by its Steiner point, $s(P_X)=0$, then
\begin{equation}\label{eq:main-H-steiner}
d_H(P_X,B)\le\frac{C}{\sqrt{t+1}}.
\end{equation}
No assumption on $N$, separation, covering radius, or spectral conditioning is required.
\end{theorem}

\begin{proof}
Let $r_0=(4\pi)^{-1/2}$ and set $\widetilde P_X=r_0P_X$, $\widetilde B=r_0B$. By \eqref{eq:prob-normalization}, their surface area measures are $\mu_X$ and $\sigma$. Both measures are centered, and by Proposition~\ref{prop:dispersion} they satisfy the uniform dispersion condition with $\vartheta=1/3$. Applying \Cref{thm:BMR} in ambient dimension $n=3$, then \eqref{eq:dc-w1} and \Cref{thm:w1}, gives
\[
\inf_a d_H(\widetilde P_X,a+\widetilde B)
\le C d_C(\mu_X,\sigma)^{1/2}
\le C W_1(\mu_X,\sigma)^{1/2}
\le \frac{C}{\sqrt{t+1}}.
\]
Scaling by $r_0^{-1}$ proves \eqref{eq:main-H}.

For the Steiner-normalized statement, recall that in $\R^3$
\[
s(K)=\frac{3}{4\pi}\int_{\Sph^2}u\,h_K(u)\,d\omega(u),
\]
so
\begin{equation}\label{eq:steiner-lip}
|s(K)-s(L)|\le3d_H(K,L).
\end{equation}
Choose $a$ such that $d_H(P_X,a+B)\le\eps$, where $\eps$ is arbitrarily close to the infimum in \eqref{eq:main-H}. Since $s(P_X)=s(B)=0$ and $s(a+B)=a$, \eqref{eq:steiner-lip} gives $|a|\le3\eps$. Therefore
\[
d_H(P_X,B)
\le d_H(P_X,a+B)+d_H(a+B,B)
\le4\eps.
\]
Letting $\eps$ decrease to the infimum proves \eqref{eq:main-H-steiner}.
\end{proof}

\begin{corollary}[Quantitative support-function convergence]\label{cor:support-convergence}
Under the Steiner normalization,
\[
\|h_{P_X}-1\|_{L^\infty(\Sph^2)}\le\frac{C}{\sqrt{t+1}}.
\]
\end{corollary}

\begin{proof}
Use \eqref{eq:hausdorff-support} and \Cref{thm:quantitative-sphericity}.
\end{proof}

\begin{corollary}[Fraenkel asymmetry]\label{cor:fraenkel}
There exists $C>0$ such that
\begin{equation}\label{eq:fraenkel-rate}
\alpha(P_X,B)\le\frac{C}{(t+1)^{3/4}}.
\end{equation}
\end{corollary}

\begin{proof}
Apply \eqref{eq:BMR-F} with $n=3$ to the probability-normalized bodies. Since Fraenkel asymmetry is invariant under common dilations,
\[
\alpha(P_X,B)^2
\le C d_C(\mu_X,\sigma)^{3/2}
\le C W_1(\mu_X,\sigma)^{3/2}
\le C(t+1)^{-3/2}.
\]
Taking square roots proves \eqref{eq:fraenkel-rate}.
\end{proof}

\begin{remark}[Universal input versus structured geometry]\label{rem:blackbox}
The inverse stability theorem is an external convex-geometric input. The contribution here is to show that spherical designs form a highly structured atomic class for which its hypotheses are satisfied automatically and uniformly, and to identify precisely which information is already forced by design exactness. The $W_1$ scale $t^{-1}$ and the dispersion bound $\Theta\ge1/3$ are intrinsic consequences of the design equations. Any improvement of the exponent $1/2$ must therefore use structure beyond the general nondegeneracy mechanism, such as high-order moments, spectral conditioning, or regularity of the normal fan.
\end{remark}

\begin{remark}[Where the exponent $1/2$ comes from]\label{rem:exponent-half}
The measure discrepancy is already of order $t^{-1}$. The loss to $t^{-1/2}$ occurs only when one inverts the nonlinear surface-area map. The exponent $1/(n-1)$ in the general quantitative Minkowski theorem is sharp, so an improvement within the design class would have to exploit its additional algebraic or spectral rigidity rather than measure nondegeneracy alone.
\end{remark}

\section{Projection bodies and the full approximation scale}\label{sec:projection}

The universal Hausdorff estimate for $P_X$ loses a square root through nonlinear inverse Minkowski stability. For the projection body the surface-area measure enters linearly, so the full $O(t^{-1})$ scale is retained.

\begin{theorem}[Uniform brightness estimate]\label{thm:brightness}
Let $X_N$ be a spherical $t$-design with Minkowski polytope $P_X$. Then for every $v\in\Sph^2$,
\begin{equation}\label{eq:brightness}
|b_{P_X}(v)-\pi|\le\frac{4\pi C_J}{t+1}.
\end{equation}
Consequently,
\begin{equation}\label{eq:projection-rate}
d_H(\Pi P_X,\pi B)\le\frac{4\pi C_J}{t+1}.
\end{equation}
\end{theorem}

\begin{proof}
By Cauchy's formula and \eqref{eq:realization},
\[
b_{P_X}(v)=2\pi\int_{\Sph^2}|u\cdot v|\,d\mu_X(u).
\]
For the unit ball,
\[
\pi=b_B(v)=2\pi\int_{\Sph^2}|u\cdot v|\,d\sigma(u).
\]
The function $u\mapsto|u\cdot v|$ is $1$-Lipschitz on the sphere. Therefore
\[
|b_{P_X}(v)-\pi|
\le2\pi W_1(\mu_X,\sigma)
\le\frac{4\pi C_J}{t+1}.
\]
Taking the supremum over $v$ and using $h_{\Pi P_X}=b_{P_X}$ proves \eqref{eq:projection-rate}.
\end{proof}

The brightness discrepancy retains the harmonic gap of the design.

\begin{proposition}[Low-frequency cancellation of brightness]\label{prop:brightness-gap}
Define
\[
\beta_X(v):=b_{P_X}(v)-\pi.
\]
Then
\begin{equation}\label{eq:brightness-gap}
\Pi_{\le t}\beta_X=0.
\end{equation}
In fact, $\beta_X$ is even and therefore has only even-degree spherical harmonics.
\end{proposition}

\begin{proof}
Let $\mathcal C$ be the spherical cosine transform
\[
(\mathcal C\nu)(v):=\int_{\Sph^2}|u\cdot v|\,d\nu(u).
\]
Then
\[
\beta_X=2\pi\mathcal C(\mu_X-\sigma).
\]
The transform $\mathcal C$ commutes with rotations, and by the Funk--Hecke theorem it acts as a scalar on each spherical harmonic space $\Harm_\ell$. Since $\mu_X-\sigma$ has zero coefficients in $\Harm_\ell$ for $1\le\ell\le t$, so does $\beta_X$. The kernel $|u\cdot v|$ is even in both variables, hence the odd-degree multipliers vanish.
\end{proof}

\begin{corollary}[Uniform control of mixed area functionals]\label{cor:mixed-area}
Let $L\in\K^3$ and suppose, after translation, that $L\subset RB$. Then
\[
|V(P_X,P_X,L)-V(B,B,L)|\le\frac{8\pi C_JR}{3(t+1)}.
\]
If $h_L\in\Pi_t(\Sph^2)$, the left-hand side is exactly zero.
\end{corollary}

\begin{proof}
This is Theorem~\ref{thm:mixed-exact} and Corollary~\ref{cor:mixed-lip}.
\end{proof}

\begin{remark}[Two geometric scales]\label{rem:two-scales}
\Cref{thm:quantitative-sphericity,thm:brightness} isolate two different mechanisms. The scale $t^{-1}$ is inherited directly from approximation of the surface-area measure and is retained by geometric observables that depend linearly on that measure. The universal $t^{-1/2}$ Hausdorff scale appears only in reconstructing the convex body itself through the nonlinear inverse Minkowski map.
\end{remark}

\section{Surface-area measures and harmonic shape modes}\label{sec:modes}

This section gives a complementary interpretation of the preceding geometry. The main results of \Cref{sec:realization,sec:isotropy,sec:quantitative,sec:projection} do not depend on local PDE analysis. Rather, the spherical Minkowski--Monge--Amp\`ere formulation explains how the exact moment conditions act on infinitesimal shape modes near the ball.

\subsection{The Aleksandrov surface-area Monge--Amp\`ere measure}
For a continuous function $h:\Sph^2\to\R$, define its Wulff shape
\begin{equation}\label{eq:wulff}
K[h]:=\bigcap_{u\in\Sph^2}\{x\in\R^3:x\cdot u\le h(u)\}.
\end{equation}
Whenever $K[h]\in\K^3_\circ$ and $h=h_{K[h]}$, define
\begin{equation}\label{eq:MA-def}
\MA_{\Sph^2}[h]:=S_{K[h]}.
\end{equation}
This is the Aleksandrov spherical Monge--Amp\`ere measure associated with the support function $h$.

If $h\in C^2(\Sph^2)$ and
\begin{equation}\label{eq:Qh}
Q_h:=\nabla^2_{\Sph^2}h+h g_{\Sph^2}>0
\end{equation}
as a quadratic form on every tangent space, then $h$ is the support function of a smooth strictly convex body and
\begin{equation}\label{eq:smooth-MA}
d\MA_{\Sph^2}[h]
=\det(\nabla^2_{\Sph^2}h+h g_{\Sph^2})\,d\omega.
\end{equation}
For nonsmooth support functions, \eqref{eq:MA-def} is the weak form of the same equation; see Schneider~\cite[Chapter~8]{Schneider2014} and Cheng--Yau~\cite{ChengYau1976}.

\begin{theorem}[Surface-area Monge--Amp\`ere characterization]\label{thm:MA-char}
Let $X_N$ be a spherical $t$-design with $t\ge2$, let $P_X$ be its Minkowski polytope, and let $h_X=h_{P_X}$. Then
\begin{equation}\label{eq:MA-design}
\MA_{\Sph^2}[h_X]=\frac{4\pi}{N}\sum_{j=1}^N\delta_{x_j}.
\end{equation}
For the unit ball $B$, $h_B\equiv1$ and
\begin{equation}\label{eq:MA-ball}
\MA_{\Sph^2}[1]=\omega.
\end{equation}
Moreover, the following are equivalent:
\begin{enumerate}[label=(\roman*)]
\item $X_N$ is a spherical $t$-design;
\item for every $p\in\Pi_t(\Sph^2)$,
\begin{equation}\label{eq:MA-moments}
\int_{\Sph^2}p\,d\MA_{\Sph^2}[h_X]
=\int_{\Sph^2}p\,d\MA_{\Sph^2}[1];
\end{equation}
\item for every $1\le\ell\le t$ and every $Y\in\Harm_\ell$,
\begin{equation}\label{eq:MA-harmonics}
\int_{\Sph^2}Y\,d\bigl(\MA_{\Sph^2}[h_X]-\MA_{\Sph^2}[1]\bigr)=0.
\end{equation}
\end{enumerate}
\end{theorem}

\begin{proof}
Equations \eqref{eq:MA-design} and \eqref{eq:MA-ball} follow from \eqref{eq:MA-def}, \Cref{thm:realization}, and $S_B=\omega$. Dividing both measures by $4\pi$ turns \eqref{eq:MA-moments} into the spherical design condition \eqref{eq:design}. The equivalence with \eqref{eq:MA-harmonics} is the standard harmonic decomposition of $\Pi_t(\Sph^2)$.
\end{proof}

\begin{remark}[Translation invariance]\label{rem:MA-translation}
Adding a linear function $a\cdot u$ to a support function translates the body. Since
\[
\nabla^2_{\Sph^2}(a\cdot u)+(a\cdot u)g_{\Sph^2}=0,
\]
the smooth operator in \eqref{eq:smooth-MA}, and hence the Aleksandrov measure $\MA_{\Sph^2}[h]$, is invariant under this addition. This is the local operator-theoretic origin of the translation ambiguity in Minkowski's theorem.
\end{remark}

\subsection{Linearization at the unit ball}
We use the sign convention
\begin{equation}\label{eq:laplace-sign}
\Delta_{\Sph^2}Y=-\ell(\ell+1)Y,\qquad Y\in\Harm_\ell.
\end{equation}

\begin{proposition}[Linearized Minkowski operator]\label{prop:linearized}
Let
\[
\mathcal M(h):=\det(\nabla^2_{\Sph^2}h+h g_{\Sph^2}).
\]
Then
\begin{equation}\label{eq:linearized}
D\mathcal M(1)[\varphi]=(\Delta_{\Sph^2}+2)\varphi.
\end{equation}
On $\Harm_\ell$ the linearized operator acts by the scalar
\begin{equation}\label{eq:linearized-spectrum}
2-\ell(\ell+1).
\end{equation}
Its kernel is exactly $\Harm_1$, the space of infinitesimal translations.
\end{proposition}

\begin{proof}
For $h_\eps=1+\eps\varphi$,
\[
\nabla^2_{\Sph^2}h_\eps+h_\eps g
=g+\eps(\nabla^2_{\Sph^2}\varphi+\varphi g).
\]
In an orthonormal tangent frame, the derivative of the determinant at the identity is the trace. Hence
\[
D\mathcal M(1)[\varphi]
=\tr(\nabla^2_{\Sph^2}\varphi+\varphi g)
=\Delta_{\Sph^2}\varphi+2\varphi.
\]
Equation \eqref{eq:linearized-spectrum} follows from \eqref{eq:laplace-sign}. The eigenvalue vanishes precisely for $\ell=1$.
\end{proof}

\begin{corollary}[Linearized suppression of low shape modes]\label{cor:mode-suppression}
Let $h_s=1+s\varphi+O(s^2)$ be a differentiable family of smooth support functions satisfying
\[
\mathcal M(h_s)=1+s f+O(s^2),
\qquad
\Pi_{\Harm_1}\varphi=0.
\]
Then
\begin{equation}\label{eq:first-response}
(\Delta_{\Sph^2}+2)\varphi=f.
\end{equation}
If $\Pi_{\le t}f=0$ for some $t\ge2$, then $\Pi_{\le t}\varphi=0$ and
\begin{equation}\label{eq:response-bound}
\|\varphi\|_{L^2(\sigma)}
\le\frac1{(t+1)(t+2)-2}\|f\|_{L^2(\sigma)}.
\end{equation}
\end{corollary}

\begin{proof}
Differentiating at $s=0$ and using Proposition~\ref{prop:linearized} gives \eqref{eq:first-response}. Expanding in spherical harmonics yields
\[
\widehat\varphi(\ell,m)
=\frac{\widehat f(\ell,m)}{2-\ell(\ell+1)},\qquad \ell\ne1.
\]
The asserted cancellation follows immediately. For $\ell\ge t+1$ one has
$|2-\ell(\ell+1)|\ge (t+1)(t+2)-2$, and Parseval's identity gives \eqref{eq:response-bound}.
\end{proof}

\begin{remark}[Affine and ellipsoidal modes]\label{rem:ellipsoidal}
The degree-zero mode changes scale, while the degree-one modes translate the body and form the kernel of the linearized operator. Degree-two harmonics are the first genuine anisotropic modes; infinitesimal volume-preserving linear images of the ball have support-function variation in $\Harm_2$. Hence a spherical design of strength $t\ge2$ eliminates, in the linearized Minkowski response, the entire ellipsoidal anisotropy. Higher design strength removes progressively higher-order shape modes.
\end{remark}

\begin{remark}[Nonlinear mode generation]\label{rem:nonlinear-generation}
Corollary~\ref{cor:mode-suppression} is an infinitesimal statement. The determinant is nonlinear, so interactions among high-frequency modes may generate lower frequencies at quadratic and higher orders. No exact low-frequency cancellation for the full support function is asserted. The universal estimates of \Cref{sec:quantitative} are independent of this linearization and are proved directly at the level of surface-area measures.
\end{remark}

\section{Affine covariance and symmetry}\label{sec:affine}

Ordinary spherical designs are invariant under $O(3)$, not under the full linear group. The Minkowski realization makes the correct affine statement transparent: surface area measures are contravariant with an explicit weight.

\begin{proposition}[Affine transformation of the Minkowski polytope]\label{prop:affine}
Let $A\in GL(3)$ and let $P_X$ have facet normals $x_j$ and facet areas $a_j=4\pi/N$. Then the corresponding facets of $AP_X$ have outer unit normals
\begin{equation}\label{eq:affine-normal}
x_j^A=\frac{A^{-T}x_j}{|A^{-T}x_j|}
\end{equation}
and areas
\begin{equation}\label{eq:affine-area}
a_j^A=|\det A|\,|A^{-T}x_j|\,a_j.
\end{equation}
Consequently,
\begin{equation}\label{eq:affine-SA}
S_{AP_X}
=\frac{4\pi|\det A|}{N}
\sum_{j=1}^N|A^{-T}x_j|\,
\delta_{A^{-T}x_j/|A^{-T}x_j|}.
\end{equation}
The transformed facet-area vectors remain balanced:
\begin{equation}\label{eq:affine-balance}
\sum_{j=1}^Na_j^Ax_j^A=0.
\end{equation}
\end{proposition}

\begin{proof}
The normal to the image of a plane with normal $x_j$ is proportional to $A^{-T}x_j$, which gives \eqref{eq:affine-normal}. The area scaling of a two-dimensional element with unit normal $x_j$ is the cofactor factor $|\det A|\,|A^{-T}x_j|$, proving \eqref{eq:affine-area}. Summation gives \eqref{eq:affine-SA}. Finally,
\[
\sum_{j=1}^Na_j^Ax_j^A
=|\det A|A^{-T}\sum_{j=1}^Na_jx_j=0.
\]
\end{proof}

\begin{remark}[Affine covariance rather than affine invariance]\label{rem:affine-covariance}
Unless $A$ is a scalar multiple of an orthogonal map, the factors $|A^{-T}x_j|$ are not constant. Therefore an equal-area Minkowski polytope is generally transformed into a weighted one, and its normalized facet-normal measure is not an ordinary equal-weight spherical design. The natural affine image is a weighted projective normal configuration. Thus the spherical design property itself is orthogonally invariant, while the surface-area-measure realization is affinely covariant.
\end{remark}

\section{Spectral conditioning and local regularity of facet normals}\label{sec:spectral}

The Minkowski realization uses only polynomial exactness. It does not by itself prevent two facet normals from being extremely close. In the critical interpolation regime, spectral conditioning supplies precisely this missing local control.

Let
\[
d_t=(t+1)^2,
\]
choose an orthonormal basis of $\Pi_t(\Sph^2)$ in $L^2(\sigma)$, and let $y_t(x)\in\R^{d_t}$ be the corresponding evaluation vector. For $X_N=\{x_1,\dots,x_N\}$ define
\begin{equation}\label{eq:gram}
Y_t=[y_t(x_1)\ \cdots\ y_t(x_N)],
\qquad
\widehat H_t:=\frac1N Y_tY_t^T.
\end{equation}
The addition formula gives
\begin{equation}\label{eq:addition}
|y_t(x)|^2=d_t.
\end{equation}

\begin{theorem}[Spectral conditioning separates facet normals]\label{thm:separation}
Let $X_N\subset\Sph^2$ be a spherical $t$-design. Assume $N=d_t$ and that $X_N$ is fundamental for $\Pi_t(\Sph^2)$. Set
\[
a_t:=\lambda_{\min}(\widehat H_t)>0.
\]
Then
\begin{equation}\label{eq:separation}
\min_{i\ne j}d_{\Sph^2}(x_i,x_j)\ge\frac{\sqrt{a_t}}{t}.
\end{equation}
Hence, for the Minkowski polytope $P_X$, distinct facet normals are separated at the wavelength scale whenever $a_t$ is bounded below.
\end{theorem}

\begin{proof}
Let $L_1,\dots,L_N$ be the Lagrange basis. With $L(x)=(L_1(x),\dots,L_N(x))^T$,
\[
\sum_{j=1}^N L_j(x)^2
=\frac1N y_t(x)^T\widehat H_t^{-1}y_t(x)
\le\frac{d_t}{Na_t}=a_t^{-1}.
\]
Thus $\|L_j\|_\infty\le a_t^{-1/2}$. For $i\ne j$, restrict $L_i$ to the great circle through $x_i$ and $x_j$. It is a trigonometric polynomial of degree at most $t$. Bernstein's inequality gives
\[
1=|L_i(x_i)-L_i(x_j)|
\le t\,d_{\Sph^2}(x_i,x_j)\|L_i\|_\infty
\le t\,d_{\Sph^2}(x_i,x_j)a_t^{-1/2}.
\]
Rearranging proves \eqref{eq:separation}.
\end{proof}

\begin{corollary}[Condition-number certificate]\label{cor:condition-number}
Under the assumptions of \Cref{thm:separation}, if $\kappa_2(\widehat H_t)\le K$, then
\begin{equation}\label{eq:condition-separation}
\min_{i\ne j}d_{\Sph^2}(x_i,x_j)\ge\frac1{t\sqrt K}.
\end{equation}
\end{corollary}

\begin{proof}
Since $\tr\widehat H_t=d_t=N$, the largest eigenvalue is at least $1$. Therefore $\lambda_{\min}\ge1/K$, and \Cref{thm:separation} applies.
\end{proof}

\begin{remark}[What spectral conditioning does and does not control]\label{rem:spectral-limits}
\Cref{thm:separation} regularizes the distribution of facet normals. Together with standard covering estimates for near-minimal spherical designs, it produces a quasi-uniform normal mesh; compare Sol\'e~\cite{Sole1991}, Brauchart--Dick--Saff--Sloan--Wang--Womersley~\cite{Brauchart2015}, and An--Chen--Sloan--Womersley~\cite{An2010}. However, a quasi-uniform set of facet normals does not by itself improve the universal $O(t^{-1/2})$ Hausdorff rate or give direct bounds for support numbers and facet aspect ratios. Any stronger shape estimate must use additional structure in the inverse Minkowski problem.
\end{remark}

\subsection{Can spectral conditioning improve the shape rate?}
The estimate of \Cref{thm:quantitative-sphericity} is universal and does not use spectral conditioning. In the critical square case $N=(t+1)^2$, a uniform spectral lower bound gives additional local geometry. Specifically,
\[
\lambda_{\min}(\widehat H_t)\ge a_0>0
\quad\Longrightarrow\quad
\delta_X\ge\frac{\sqrt{a_0}}{t}.
\]
Thus well-conditioned critical designs have facet normals separated at the natural wavelength scale. Together with known covering estimates, this produces quasi-uniform normal meshes. This information is strictly stronger than the dispersion bound $\Theta(\mu_X)\ge1/3$ used in the general inverse Minkowski theorem. It is therefore natural to ask whether the H\"older loss in \Cref{thm:quantitative-sphericity} can be reduced on this restricted class.

\begin{conjecture}[Spectrally improved Minkowski rate]\label{conj:spectral-rate}
Let $X_t$ be critical spherical $t$-designs with $N=(t+1)^2$ and assume
\[
\lambda_{\min}(\widehat H_t)\ge a_0>0
\]
uniformly in $t$. Let $P_t$ be the Steiner-normalized Minkowski polytopes. Then there exists $\beta>1/2$ such that
\begin{equation}\label{eq:spectral-conj}
d_H(P_t,B)\le Ct^{-\beta}.
\end{equation}
A natural first target is the wavelength-scale rate $\beta=1$.
\end{conjecture}

The conjecture cannot be obtained merely by strengthening a Marcinkiewicz--Zygmund inequality for arbitrary $L^\infty$ functions. The required improvement must enter the inverse Minkowski step, for example through quantitative control of the normal fan, regularity of the support numbers, or a stability theorem adapted to quasi-uniform atomic surface area measures. In this formulation, the open problem is a refinement of the convex geometry of the structured atomic class generated by spherical designs.

\section{Higher-dimensional extension and open problems}\label{sec:higher}

The same geometric mechanism extends directly to $\Sph^d\subset\R^{d+1}$. Let $\sigma_d$ be normalized area measure and let $\omega_d=|\Sph^d|$.

\begin{theorem}[Higher-dimensional realization and quantitative sphericity]\label{thm:higher}
Let $X_N\subset\Sph^d$ be a spherical $t$-design with $t\ge2$. Then there exists a full-dimensional convex polytope $P_X\subset\R^{d+1}$, unique up to translation, such that
\begin{equation}\label{eq:higher-SA}
S_{P_X}=\frac{\omega_d}{N}\sum_{x\in X_N}\delta_x.
\end{equation}
All facets have equal $d$-dimensional area. Moreover, after fixing translations by the Steiner point,
\begin{equation}\label{eq:higher-H}
d_H(P_X,B^{d+1})\le C_d(t+1)^{-1/d}.
\end{equation}
The Fraenkel asymmetry satisfies
\begin{equation}\label{eq:higher-F}
\alpha(P_X,B^{d+1})\le C_d(t+1)^{-(d+1)/(2d)}.
\end{equation}
\end{theorem}

\begin{proof}
Degree-one exactness gives balance, while degree-two exactness gives
\begin{equation}\label{eq:higher-second}
\frac1N\sum_{x\in X_N}xx^T=\frac1{d+1}I_{d+1}.
\end{equation}
Hence the classical Minkowski theorem gives \eqref{eq:higher-SA}. Moreover, for each $\theta\in\Sph^d$,
\[
\int|\theta\cdot u|\,d\mu_X(u)
\ge\int(\theta\cdot u)^2\,d\mu_X(u)
=\frac1{d+1},
\]
so $\Theta(\mu_X)\ge1/(d+1)$. A dimension-dependent spherical Jackson inequality gives
\[
W_1(\mu_X,\sigma_d)\le C_d(t+1)^{-1}.
\]
Set $r_d=\omega_d^{-1/d}$. Since surface area measure in $\R^{d+1}$ scales by $S_{rK}=r^dS_K$, one has $S_{r_dP_X}=\mu_X$ and $S_{r_dB^{d+1}}=\sigma_d$. Apply \Cref{thm:BMR} in ambient dimension $n=d+1$. The Hausdorff exponent is $1/(n-1)=1/d$, yielding \eqref{eq:higher-H}. The Fraenkel estimate has squared exponent $1+1/d$, so after taking square roots one obtains $(d+1)/(2d)$ in \eqref{eq:higher-F}. Steiner normalization removes the remaining translation as in \Cref{thm:quantitative-sphericity}.
\end{proof}

The theorem separates a dimension-stable approximation mechanism from a dimension-dependent nonlinear inversion mechanism:
\[
\boxed{
\text{design exactness}
\Longrightarrow
W_1=O(t^{-1})
\Longrightarrow
 d_H=O(t^{-1/d}).}
\]
The first arrow reflects spherical approximation of the curvature measure; the dimension loss enters only in the second arrow, through inverse Minkowski stability.

Several problems now become sharply formulated.
\begin{enumerate}[label=(\roman*)]
\item \textbf{Sharpness within the design class.} The exponent $1/d$ is sharp for general uniformly dispersed surface area measures, but spherical designs satisfy exact vanishing of all harmonics through degree $t$. Determine the optimal exponent for this special class.
\item \textbf{Spectral improvement.} Determine whether critical well-conditioned designs satisfy the wavelength-scale estimate $d_H(P_X,B)=O(t^{-1})$, or more generally any exponent strictly larger than $1/d$.
\item \textbf{Structured inverse Minkowski stability.} Identify geometric hypotheses on the normal fan, support numbers, or facet regularity that convert wavelength-scale regularity of the normals into a sharper shape rate.
\item \textbf{Harmonic mode transfer.} For smooth regularizations of the atomic design measures, quantify how nonlinear interactions create low-frequency modes beyond the linearized equation $(\Delta+d)\varphi=f$. This is an interpretative refinement of the surface-area viewpoint rather than a prerequisite for the universal shape theorem.
\end{enumerate}

A useful quantitative invariant is the translation-normalized shape deficit
\begin{equation}\label{eq:shape-deficit}
M_t(X):=d_H(P_X-s(P_X),B).
\end{equation}
The present paper proves the universal estimate
\begin{equation}\label{eq:shape-deficit-bound}
M_t(X)\le C(t+1)^{-1/2}
\end{equation}
on $\Sph^2$. Improving this rate under additional geometric or spectral hypotheses is the natural next step.

\section{Conclusion}

This paper develops a quantitative convex-geometric interpretation of spherical designs. Every spherical $t$-design of strength at least two determines a unique translation class of equal-facet-area convex polytopes, and under this correspondence
\[
\frac{S_{P_X}}{4\pi}=\mu_X.
\]
The design equations therefore become exact low-degree moment conditions for a discrete surface-area measure. Their first consequence is not merely qualitative existence: they force the normal geometry of $P_X$ to agree with that of the ball through order $t$. The normalized surface tensors satisfy $M_k(P_X)=M_k(B)$ for $k\le t$, and mixed volumes are exact against polynomial support functions of degree at most $t$. Thus spherical design strength acquires a direct geometric meaning as high-order isotropy of an equal-area polytope.

The same exactness produces universal quantitative shape control. Jackson approximation gives $W_1(\mu_X,\sigma)=O(t^{-1})$, while degree-two isotropy gives the uniform dispersion bound $\Theta(\mu_X)\ge1/3$. Quantitative inverse Minkowski stability then yields
\[
d_H(P_X-s(P_X),B)=O(t^{-1/2})
\]
for every spherical $t$-design on $\Sph^2$, independently of cardinality and conditioning, together with the Fraenkel estimate $\alpha(P_X,B)=O(t^{-3/4})$. For projection bodies, where the surface-area measure enters linearly, the full approximation scale is retained:
\[
d_H(\Pi P_X,\pi B)=O(t^{-1}).
\]
The contrast between these two rates separates the approximation scale of the surface-area data from the loss caused by nonlinear reconstruction of the convex body.

Spectral conditioning supplies a second layer of geometry. In the critical interpolation regime, a uniform lower spectral bound forces facet normals to be separated at the wavelength scale $t^{-1}$. Hence algebraic exactness and spectral regularity play different roles: the former gives universal global sphericity, while the latter controls local geometry. The open problem is to determine whether combining these structures yields a sharper inverse Minkowski rate for the structured atomic measures arising from spherical designs.

The surface-area Monge--Amp\`ere formulation provides a complementary interpretation of this picture. Its linearization at the ball identifies translations and ellipsoidal anisotropy as low harmonic shape modes and explains why design exactness is naturally read as high-order isotropy. The main framework, however, is convex-geometric:
\[
\boxed{
\text{spherical harmonic exactness}
\longrightarrow
\text{exact surface-area geometry}
\longrightarrow
\text{quantitative sphericity}.}
\]
This viewpoint places spherical designs at a natural interface between approximation theory, harmonic analysis, and convex geometry, and suggests a broader program of studying discrete designs through the geometry of the convex bodies determined by their moment measures.

\section*{Statements and Declarations}
\noindent\textbf{Competing interests.} The author declares no competing interests related to this manuscript.

\noindent\textbf{Data availability.} No datasets were generated or analyzed in this theoretical study.

\noindent\textbf{Use of generative AI.} During preparation of this manuscript, the author used Deepseek for language editing, organization, and drafting assistance. The author critically reviewed and revised the resulting text and takes full responsibility for the mathematical content and the final manuscript.

\section*{Acknowledgments}
The author thanks Yeyao Hu and Yingnan Wang for encouragement and helpful discussions.

\end{document}